\documentclass[12pt]{amsart}

\usepackage{amssymb,amsfonts,amsthm}
\usepackage{amsmath}
\usepackage[all]{xy}

\usepackage[margin=3cm]{geometry}

\numberwithin{equation}{section}
\newtheorem{conj}[equation]{Conjecture}

\newtheorem{thm}[equation]{Theorem}
\newtheorem{lem}[equation]{Lemma}
\newtheorem{cor}[equation]{Corollary}
\newtheorem{prop}[equation]{Proposition}

\theoremstyle{definition}
\newtheorem{dfn}[equation]{Definition}
\newtheorem{exam}[equation]{Example}

\theoremstyle{remark}
\newtheorem{rem}[equation]{Remark}

\renewcommand{\leq}{\leqslant}
\renewcommand{\geq}{\geqslant}

\newcommand{\abs}[1]{\left\lvert#1\right\rvert}

\newcommand{\norm}[1]{\left\|#1\right\|}
\newcommand{\MA}{\mathop{\mathrm{MA}}\nolimits}
\newcommand{\Aut}{\mathop{\mathrm{Aut}}\nolimits}

\newcommand{\PSH}{\mathop{\mathrm{PSH}}\nolimits}

\newcommand{\vol}{\mathop{\mathrm{vol}}\nolimits}

\begin{document}

%Title%
\title[]{On the limit of spectral measures associated to a test configuration}
\author[]{Tomoyuki Hisamoto}
\address{Graduate School of Mathematical Sciences, The University of Tokyo, 
3-8-1 Komaba Meguro-ku, Tokyo 153-0041, Japan}
\email{hisamoto@ms.u-tokyo.ac.jp}
\date{}
\maketitle

%Abstract%
\begin{abstract}
We apply the result of \cite{His12} to the family of graded linear series constructed from any test configuration. This solves the conjecture raised by \cite{WN10} so that the sequence of spectral measures for the induced $\mathbb{C}^*$-action on the central fiber converges to the canonical Duistermatt--Heckman measure defined by the associated weak geodesic ray. As a consequence, we show that the algebraic $p$-norm of the test configuration equals to the $L^p$-norm of tangent vectors. Using this result, We may give a natural energy theoretic explanation for the lower bound estimate on the Calabi functional by \cite{Don05} and prove the analogous result for the K\"{a}hler--Einstein metric. 
\end{abstract}

%Introduction%
\section{Introduction}\label{introduction} 
Let $X$ be an $n$-dimensional smooth projective variety and $L$ an ample line bundle over $X$. In the sequel we also fix a smooth Hermitian metric $h$ on $L$, which has strictly positive curvature over $X$. The curvature form defines a K\"{a}hler metric in the first Chern class $c_1(L)$. Conversely, any K\"{a}hler metric $\omega$ in $c_1(L)$ has a K\"{a}hler potential $\varphi$ in each local trivialization neighborhood such that the correction of $e^{-\varphi}$ defines a Hermitian metric with the curvature form $\omega=dd^c\varphi$, uniquely up to multiplication of a constant. We identify $h$ with the correction of weights $\varphi$ and with $\omega$. The space of K\"{a}hler metric $\mathcal{H}$ is the set of all $h=e^{-\varphi}$, endowed with the canonical Riemannian metric 
\begin{equation*}
\norm{u}_2:=\bigg(\int_X u^2 \frac{(dd^c\varphi)^n}{n!}\bigg)^{\frac{1}{2}}
\end{equation*} 
which is defined for any tangent vector $u$ at $\varphi$. 
There is the canonical K-energy functional $\mathcal{M}\colon \mathcal{H} \to \mathbb{R}$ such that any constant scalar curvature K\"{a}hler metric is characterized as a critical point of this energy. This K-energy is known to be convex along any smooth geodesic ray in $\mathcal{H}$ and it is important to investigate the gradient of the energy at infinity along a given geodesic ray $\varphi_t$ ($t\in[0, +\infty)$).  

A geodesic ray in $\mathcal{H}$ corresponds to a special kind of degeneration of $(X, L)$ in algebraic geometry. A flat family of polarized schemes $\pi: (\mathcal{X}, \mathcal{L}) \to \mathbb{C}$ with $(\mathcal{X}_1, \mathcal{L}_1)=(X, L)$ and an equivariant $\mathbb{C}^*$-action on $(\mathcal{X}, \mathcal{L})$ is called a test configuration. We denote the datum by $\mathcal{T}$. For each $k\geq1$ let $H^0(\mathcal{X}_0, \mathcal{L}_0^{\otimes k})=\bigoplus_{\lambda} V_{\lambda}$ be the eigenspace decomposition of the induced $\mathbb{C}^*$-action on the central fiber. Then we have the asymptotic expansion 
\begin{equation*}
\frac{\sum_{\lambda} \lambda \dim V_{\lambda}}{k \sum_{\lambda} \dim V_{\lambda}}
= F_0 + F_1 k^{-1} + O(k^{-2}). 
\end{equation*}
We call the coefficient $F_1$ in the subleading term as the Donaldson--Futaki invariant of $\mathcal{T}$. 
It was first established by \cite{PS07} that any test configuration $\mathcal{T}$ with fixed metric $\varphi$ canonically defines a {\em weak} geodesic ray $\varphi_t$ emanating from $\varphi$, in the space of K\"{a}hler metric. (Here for the proof of the main theorem we adopt the construction of \cite{RWN11} so that $\varphi_t -F_0$ gives the geodesic ray in \cite{PS07}.) In this situation it is now conjectured that the Donaldson--Futaki invariant corresponds to $\lim_{t\to\infty}\frac{d}{dt}\mathcal{M}(\varphi_t)$ if the latter one is properly defined for the weak geodesic ray. In this paper we further relate the asymptotic {\em distribution} of eigenvalues to $\varphi_t$ and give some application to the estimate for $F_1$. Our main theorem claims that the associated sequence of spectral measures converges to the canonical Duistermatt--Heckman measure defined by $\varphi_t$. The Monge--Amp\`{e}re (or Liouville) measure $\MA(\varphi_t)$ is defined for each singular $\varphi_t$ and equals to $(dd^c\varphi_t)^n$ if $\varphi_t$ is smooth (see subsection $2.1$). 
\begin{thm}\label{main2}
Let $\mathcal{T}$ be a test configuration with normal $\mathcal{X}$. 
Then the weak limit of the normalized distribution of eigenvalues is given by the push-forward of the Monge--Amp\`{e}re measure $\MA(\varphi_t)$ to the real line by the tangent vector $\dot{\varphi}_t$. That is, for any $t\geq0$ we have 
\begin{equation*}
\lim_{k\to \infty}\frac{n!}{k^n} \sum_{\lambda} \delta_{\frac{\lambda}{k}}\dim V_{\lambda} = (\dot{\varphi}_t)_*\MA(\varphi_t).  
\end{equation*}
Here $\delta_{\frac{\lambda}{k}}$ denotes the delta function for $\frac{\lambda}{k}\in\mathbb{R}$. 
In particular, the right hand side measure is independent not only of $t$ but also of $\varphi$, and defines the canonical Duistermatt--Heckman measure. 
\end{thm}
Theorem \ref{main2} was first conjectured by \cite{WN10} and proved for product test configurations in the same paper. The analogous result for geodesic segments was obtained by \cite{Bern09} in a different approach. 
Recall that the above definition of $F_1$ was motivated by the equivariant Riemann--Roch formula of \cite{AB84}, which can be applied to the product test configuration and in that case one has the Duistermatt--Heckman measure on the central fiber in the usual way. In word of geodesic the central fiber corresponds to $t= \infty$ and our canonical Duistermatt--Heckman measure which is independent of $t$ gives the right generalization to any test configuration. Then Theorem \ref{main2} can be seen as a part of the {\em ideal} index theorem for an equivariant family which admits very singular fiber over the fixed point $0\in\mathbb{C}$. Taking the $p$-th moment of the above measure, we may extend the definition of algebraic norm in \cite{Don05} to any $p\geq 1$ and relate it to the $L^p$-norm of tangent vectors on the weak geodesic ray. 
\begin{thm}\label{main3}
Let us define the trace-free part of each eigenvalue $\lambda$ as 
\begin{equation*}
\bar{\lambda}:= \lambda- \frac{\sum \lambda\dim V_{\lambda}}{\sum \dim V_{\lambda}}
\end{equation*}
and for each $p\geq 1$ define the $p$-norm $\norm{\mathcal{T}}_p$ by 
\begin{equation*}
\norm{\mathcal{T}}_p := \bigg(\lim_{k\to \infty} \frac{1}{k^n}\sum_{\lambda} \abs{\frac{\bar{\lambda}}{k}}^p \dim V_{\lambda}\bigg)^{\frac{1}{p}}. 
\end{equation*}
Then the limit exists and 
\begin{equation*}
\norm{\mathcal{T}}_p=\bigg(\int_X \abs{\dot{\varphi}_t-F_0}^p\frac{\MA(\varphi_t)}{n!} \bigg)^{\frac{1}{p}}
\end{equation*}
holds. 
\end{thm}

Using Theorem \ref{main3}, we may give an energy theoretic explanation for the \cite{Don05}'s lower bound estimate on the Calabi functional. In particular in the Fano case we obtain the following. Note that when $L=-K_X$ any metric $h=e^{-\varphi}$ can be identified with the positive measure which is described as $e^{-\varphi}\bigwedge_{i=1}^n \frac{\sqrt{-1}}{2} dz_i\wedge d\bar{z}_i$ in each local coordinate. A metric $e^{-\varphi}$ is called a K\"{a}hler--Einstein metric if it satisfies the identity of the measures: $(dd^c\varphi)^n=n!e^{-\varphi}$. 
\begin{thm}\label{main4}
Let $X$ be a Fano manifold and $\mathcal{T}$ a test configuration of $(X, -K_X)$, whose total space $\mathcal{X}$ is normal. Then for any smooth Hermitian metric $h=e^{-\varphi}$ on $-K_X$ and exponents $1\leq p, q\leq +\infty $ with $1/p+1/q=1$ we have 
\begin{equation*}
\norm{\frac{n!e^{-\varphi}}{(dd^c\varphi)^n}-1}_q \geq \frac{F_1}{\norm{\mathcal{T}}_p}.  
\end{equation*}
In other word, the difference from K\"{a}hler--Einstein metric is bounded from below by the Donaldson--Futaki invariant.  
\end{thm}

Let us briefly explain the outline of our proof of Theorem \ref{main2}. 
The proof is based on the analytic study for graded linear series, which was exploited in \cite{His12}. We apply it to \cite{WN10}'s family of graded subalgebras 
\begin{equation*}
W_{\lambda}=\bigoplus_{k=0}^{\infty} W_{\lambda, k} \subseteq \bigoplus_{k=0}^{\infty} H^0(X, L^{\otimes k}), 
\end{equation*}
which is constructed from $\mathcal{T}$ and each $\lambda \in \mathbb{R}$ as follows. For a given section $s\in H^0(X, L^{\otimes k})$, let us denote its unique invariant extension which is at least meromorphic over $\mathcal{X}$ by $\bar{s}$. We define $W_{\lambda, k}$ as the set of sections $s$ whose invariant extensions $\bar{s}$ have poles along the central fiber $\mathcal{X}_0=\{t=0\}$ at most $-\lceil \lambda k \rceil$ order. In other words, 
\begin{equation*}
W_{\lambda, k} := \bigg\{ \ s\in H^0(X, L^{\otimes k}) \ \bigg| \ t^{-\lceil \lambda k \rceil}\bar{s} \in H^0(\mathcal{X}, \mathcal{L}) \bigg\}. 
\end{equation*} 
Then it can be proved algebraically that the limit of spectral measures is given by the Lebesgue--Stieltjes measure of the volume function $\vol(W_{\lambda})$ in $\lambda$. The main theorem of \cite{His12} interpret each volume into the Monge--Amp\`{e}re measure of associated equilibrium metric $P_{W_{\lambda}}\varphi$. The Legendre transformation of this family of equilibrium metrics is nothing but the weak geodesic ray $\varphi_t$ so that we may complete the proof by the recent developed techniques of pluripotential theory. 
This kind of approach seems itself interesting and we hope it should be studied more in the future.

%Graded linear series and analytic description of the volume%
\section{Analytic description of the volume}\label{Analytic description of the volume}

\subsection{Monge--Amp\`{e}re operator}\label{Monge Ampere operator}

In this section, we briefly review the definition and basic properties of the Monge--Amp\`{e}re operator. 
Let $L$ be a holomorphic line bundle on a projective manifold $X$. We usually fix a family of local trivialization patches $U_{\alpha}$ which cover $X$. A singular Hermitian metric $h$ on $L$ is by definition a family of functions $h_{\alpha}=e^{-\varphi_{\alpha}}$ which are defined on corresponding $U_{\alpha}$ and satisfy the transition rule: $\varphi_{\beta}=\varphi_{\alpha}-\log \abs{g_{\alpha \beta}}^2$ on $U_{\alpha}\cap U_{\beta}$. Here $g_{\alpha \beta}$ are the transition functions of $L$ with respect to the indices $\alpha$ and $\beta$. {\em The weight functions} $\varphi_{\alpha}$ are assumed to be locally integrable. If $\varphi_{\alpha}$ are smooth, $\{e^{-\varphi_{\alpha}}\}_{\alpha}$ defines a smooth Hermitian metric on $L$. We usually denote the family $\{ \varphi_{\alpha} \}_{\alpha}$ by $\varphi$ and omit the indices of local trivializations. Notice that each $\varphi=\varphi_{\alpha}$ is only a local function and not globally defined, but the curvature current $\Theta_h=dd^c \varphi$ is globally defined and is semipositive if and only if each $\varphi$ is plurisubharmonic ({\em psh} for short). Here we denote by $d^c$ the real differential operator $\frac{\partial-\bar{\partial}}{4\pi\sqrt{-1}}$.  We call such a weight a {\em psh weight}. The most important example is those of the form $k^{-1}\log (\abs{s_1}^2+\cdots +\abs{s_N}^2)$, defined by some holomorphic sections $s_1, \cdots, s_N  \in H^0(X, L^{\otimes k})$. Here $\abs{s_i}$ ($1\leq i \leq N$) denotes the absolute value of the corresponding function of each $s_i$ on $U_{\alpha}$. We call such weights {\em algebraic singular}. More generally, a psh weight $\varphi$ is said to have a {\em small unbounded locus} if the pluripolar set $\varphi^{-1}(-\infty)$ is contained in some closed complete pluripolar subset $S \subset X$ ({\em e.g.} a proper algebraic subset). 
 
Let $n$ be the dimension of $X$. The Monge--Amp\`{e}re operator is defined by   
\begin{equation*}
   \varphi \mapsto  (dd^c \varphi)^n 
\end{equation*}
when $\varphi$ is smooth. On the other hand it does not make sense for general $\varphi$. 
The celebrated result of Bedford--Taylor \cite{BT76} tells us 
that the right hand side can be defined as a current 
if $\varphi$ is at least in the class $L^{\infty}\cap \PSH(U_{\alpha})$. 
Specifically, by induction on the exponent $q = 1,2,...,n$, it can be defined as: 
\begin{equation*}
   \int_{U_{\alpha}} (dd^c \varphi)^q \wedge \eta :=  
   \int_{U_{\alpha}} \varphi (dd^c \varphi)^{q-1} \wedge dd^c \eta   
\end{equation*}
for each test form $\eta$.  
Here $ \int$ denotes the canonical pairing of currents  and test forms. 
This is indeed well-defined and defines a closed positive current, 
because $ \varphi $ is a bounded Borel function 
and $ (dd^c \varphi)^{q-1} $ has measure coefficients by the induction hypothesis. Notice the fact that any closed positive current has measure coefficients. 

It is also necessary to consider unbounded psh weights. 
On the other hand, for our purpose, 
it is enough to deal with weights with small unbounded loci. 

\begin{dfn}
 Let $\varphi$ be a psh weight of a singular metric on $L$. If $\varphi$ has a small unbounded locus contained in an algebraic subset $S$, we define a positive measure $\MA(\varphi)$ on $X$ by 
 \begin{equation*}
 \MA(\varphi)  := 
 \text{ the zero extension of } ( dd^c \varphi)^n . 
 \end{equation*}
 Note that the coefficient of $( dd^c \varphi )^n$ is well-defined 
 as a measure on $X \setminus S$. 
\end{dfn}

$\MA(\varphi)$ in fact defines a closed positive current on $X$. 
In particular, it has a finite mass on $X$. 
For a proof, see \cite{BEGZ10}, Section 1. 

\begin{rem}
In \cite{BEGZ10}, the {\em non-pluripolar} Monge--Amp\`{e}re product was defined 
in fact for general psh weights on a compact K\"{a}hler manifold.
Note that this definition of the Monge--Amp\`{e}re operator makes 
the measure $\MA(\varphi)$ to have no mass on any pluripolar set.  
In other words, $\MA(\varphi)$ ignores the mass which comes from the singularities of $\psi$. 
For this reason, as a measure-valued function in $\varphi$, $\MA(\varphi)$ no longer has the continuous property. 
\end{rem}

We recall the fundamental fact established in \cite{BEGZ10} 
which states that 
the less singular psh weight has the larger Monge--Amp\`{e}re mass. 
Recall that given two psh weight $\varphi$ and $\varphi'$ on $L$, 
$\varphi$ is said to be less singular than $\varphi'$ 
if there exists a constant $C > 0$ such that $\varphi' \leq \varphi + C$ holds on $X$. 
We say that a psh weight is {\em minimal singular} 
if it is minimal with respect to this partial order. 
When $\varphi$ is less singular than $\varphi'$ and $\varphi'$ is less singular than $\varphi$, 
we say that the two functions have the equivalent singularities. 
This defines a equivalence relation. 

\begin{thm}[\cite{BEGZ10}, Theorem 1.16.]\label{comparison theorem} 
Let $\varphi$ and $\varphi'$ be psh weights 
with small unbounded loci such that
$\varphi$ is less singular than $\varphi'$. Then 
\begin{equation*}
\int_X \MA(\varphi')
\leq 
\int_X \MA(\varphi)
\end{equation*}
holds. 
\end{thm}

\subsection{Analytic representation of volume}\label{Analytic representation of volume}

Let $X$ be a $n$-dimensional smooth complex projective variety and $L$ a holomorphic line bundle on $X$. Graded linear series is by definition a graded $\mathbb{C}$-subalgebra of the section ring 
\begin{equation*}
W=\bigoplus_{k=0}^{\infty} W_k \subseteq \bigoplus_{k=0}^{\infty}H^0(X, L^{\otimes k}). 
\end{equation*} 
They appear in many geometric situations. In fact in the present paper we give an application of the analysis of such proper subalgebras to the problem of constant scalar curvature K\"{a}hler metric. The volume of graded linear series is the nonnegative real number which measures the size of the graded linear series as follows: 
\begin{equation*}
\vol(W):= \limsup_{k\to \infty} \frac{\dim W_k}{k^n/n!}. 
\end{equation*}
The main result of \cite{His12} gives an analytic description of the volume. The analytic counterpart of the volume is the following generalized equilibrium metric, which is originated from \cite{Berm07}. 
\begin{dfn}
Let $W$ be a graded linear series of a line bundle $L$. Fix a smooth Hermitian metric of $L$ and denote it by $h=e^{-\varphi}$, where $\varphi$ is the weight function defined on a fixed local trivialization neighborhood. We define the equilibrium weight associated to $W$ and $\varphi$ by 
\begin{equation*}
P_W\varphi:= {\sup}^* \bigg\{ \ \frac{1}{k}\log \abs{s}^2 \ \bigg| \ k \geq 1, \ s\in W_k \ \text{such that} \ \abs{s}^2e^{-k\varphi} \leq 1. \bigg\}. 
\end{equation*}
Here $*$ denotes taking the upper semicontinuous regularization of the function. 
The equilibrium weight $P_W\varphi$ on each local trivialization neighborhood patches together and define the singular Hermitian metric on $L$. We call it an equilibrium metric. 
\end{dfn}
As in the subsection $2.1$, we define the Monge--Amp\`{e}re measure $\MA(P_W\varphi)$ on $X$. 
\begin{thm}[The main theorem of \cite{His12}]\label{main1}
Let $W$ be a graded linear series of a line bundle $L$, such that the natural map $X\dashrightarrow \mathbb{P}W_k^*$ is birational to the image for any sufficiently large $k$. Then for any fixed smooth Hermitian metric $h=e^{-\varphi}$ we have 
\begin{equation*}
\vol(W) = \int_X \MA(P_W\varphi). 
\end{equation*}
\end{thm}

Note that Theorem \ref{main1} is valid for general line bundle which is possibly not ample. We will apply this general formula to the special graded linear series associated to a test configuration of a polarized manifold. 

\begin{rem} 
With no change of the proof in \cite{His12}, Theorem \ref{main1} can also be proved under the assumption $W_k$ is birational for any sufficiently {\em divisible} $k$. For non complete linear series, the condition $\vol(W)>0$ does not imply that $X \dashrightarrow \mathbb{P}W_k^*$ is birational for sufficiently large $k$. For example, when $W$ is defined as the pull-back of $H^0(Y, \mathcal{O}(k))$ by a finite morphism $X \to Y \subseteq \mathbb{P}^N$, $\vol(W)>0$ holds but $W_k$ never defines birational map to the image for any $k$. For this reason, neither does Theorem \ref{main1} hold for general $W$ with $\vol(W)>0$. To be precise, taking a resolution $\mu_k$ of the base ideal of $W_k$ and denoting the fixed component of $\mu_k^*W_k$ by $F_k$, the right hand side in Theorem \ref{main1} is given by the limit of self-intersection number of line bundles $M_k:= \mu_k^*L^{\otimes k}\otimes\mathcal{O}(-F_k)$. 
\end{rem}

%Test configuration and associated family of graded linear series%
\section{Test configuration and associated family of graded linear series}\label{Test configuration and associated family of graded linear series}

In this section we explain the construction of the family of graded linear series $W_{\lambda}$ $(\lambda \in \mathbb{R})$ from fixed test configuration $(\mathcal{X}, \mathcal{L})$, following the recipe of Witt-Nystr\"{o}m's paper \cite{WN10}. First we introduce the notion of K-stability. 

%Subsection%
\subsection{K-stability}\label{K-stability}

\begin{dfn}[The definition of test configuration by \cite{Don02}]\label{test}
Let $(X, L)$ be a polarized manifold. We call the following datum {\em a test configuration} ({\em resp}. semi test configuration) for $(X, L)$. 
\begin{itemize}
\item[$(1)$]
A flat family of schemes with relatively ample ({\em resp.} semiample and big) $\mathbb{Q}$-line bundle $\pi \colon (\mathcal{X}, \mathcal{L}) \to \mathbb{C}$ such that $(\mathcal{X}_1, \mathcal{L}_1) \simeq (X, L)$ holds. 
\item[$(2)$]
A $\mathbb{C}^*$-action on $(\mathcal{X}, \mathcal{L})$ which makes $\pi$ equivariant, with respect to the canonical action of $\mathbb{C}^*$ on the target space $\mathbb{C}$. 
\end{itemize}
\end{dfn}

\begin{rem}
As it was pointed out by \cite{LX11}, the above original definition by Donaldson should be a bit modified. For example, if one further assume $\mathcal{X}$ is normal, then the pathological example in \cite{LX11} can be removed. On the other hand, recent paper \cite{Sz11} proposed to consider the class of test configurations whose {\em norms} $\norm{\mathcal{T}}$ are non-zero and this condition seems more natural and appropriate for our viewpoint. In fact Theorem \ref{main3} gives one evidence. See also \cite{RWN11}. We assume, however, the normality of $\mathcal{X}$ in proving Theorem \ref{main2} and \ref{main3}. 
\end{rem}

By the flatness of $\pi$ Hilbert polynomials of $(\mathcal{X}_t, \mathcal{L}_t)$ are independent of $t \in \mathbb{C}$. The $\mathbb{C}^*$-equivariance yields an isomorphism $(\mathcal{X}_t, \mathcal{L}_t) \simeq (X, L)$ for any $t \in \mathbb{C}\setminus \{0\}$. Note that the central fiber $(\mathcal{X}_0, \mathcal{L}_0)$ can be very singular. It is even not normal in general. A test configuration is said to be {\em product} if $\mathcal{X} \simeq X \times \mathbb{C}$ and {\em trivial} if further the action of $\mathbb{C}^*$ on $X \times \mathbb{C}$ is trivial. 
A test configuration $(\mathcal{X}, \mathcal{L})$ induces the $\mathbb{C}^*$-action on $H^0(\mathcal{X}_0, \mathcal{L}_0^{\otimes k})$ for each $k \geq1$.  
This action $\rho: \mathbb{C^*} \to \Aut(H^0(\mathcal{X}_0, \mathcal{L}_0^{\otimes k}))$ decomposes the vector space as $H^0(\mathcal{X}_0, \mathcal{L}_0^{\otimes k})= \bigoplus_{\lambda} V_{\lambda}$ such that $\rho(\tau) v = \tau^{\lambda} v$ holds for any $v \in V_{\lambda}$ and $\tau \in \mathbb{C}^*$. 
By the equivariant Riemann--Roch Theorem, the total weight $w(k)=\sum_{\lambda} \lambda \dim V_{\lambda}$ is a polynomial of degree $n+1$. Let us denote the coefficients by
\begin{equation}
w(k) = b_0k^{n+1}+b_1k^n+O(k^{n-1}). 
\end{equation}
We also write the Hilbert polynomial of $(X, L)$ by 
\begin{equation*}
N_k:=\dim H^0(X, L^{\otimes k})=a_0k^n+a_1k^{n-1}+O(k^{n-2}). 
\end{equation*}
The Donaldson--Futaki invariant of given test configuration is defined to be the subleading term of 
\begin{equation}
\frac{w(k)}{kN_k}=F_0+F_1k^{-1}+O(k^{-2}). 
\end{equation}
In other word, 
\begin{equation}
F_1:= \frac{a_0b_1-a_1b_0}{a_0^2}. 
\end{equation}
\begin{dfn}
A polarization $(X, L)$ is said to be {\em K-stable} ({\em resp.\ K-semistable}) if $F_1 <0$ ({\em resp.} $c_1 \leq 0$) holds for any non-trivial test configuration. We say $(X, L)$ is {\em K-polystable} if it is K-semistable and $F_1=0$ holds only for product test configuration. 
\end{dfn}
This notion of K-stability was first introduced by \cite{Tia97}. The above algebraic definition was given by \cite{Don02}. Note that K-stability is unchanged if one replaces $L$ to $L^{\otimes k}$ since $F_1$ is so. The equivalence of certain GIT-stability and existence of special metric originates from Kobayashi--Hitchin correspondence for vector bundles. In the polarized manifolds case, we have the following conjecture. 

\begin{conj}$($Yau--Tian--Donaldson$)$
A polarized manifold $(X, L)$ admits a cscK metric if and only if it is K-polystable.  
\end{conj}

One direction of the above conjecture was proved by \cite{Don05}, \cite{Sto09}, \cite{Mab08}, \cite{Mab09}. That is, the existence of cscK metric implies K-polystability of the polarized manifold. See also \cite{Berm12} for the detail study in the K\"{a}hler--Einstein case.

The stability of vector bundle is defined by {\em slope} of subbundles and 
to pursue the analogy to the vector bundle case, \cite{RT06} studied the special type of test configurations which are defined by subschemes of $X$, and introduced the slope of a subscheme. 

\begin{exam}\label{RT}
A pair of an ideal sheaf $\mathcal{J} \subseteq \mathcal{O}_X$ and $c \in \mathbb{Q}$ defines a test configuration as follows. We call such test configuration as {\em deformation to the normal cone} with respect to $(\mathcal{J}, c)$: Let $\mathcal{X}$ be the blow-up of $X \times \mathbb{C}$ along $\mathcal{J}$ and $P$ be the exceptional divisor. The action of $\mathbb{C}^*$ on $X \times \mathbb{C}$ fix $V(\mathcal{J})$ so that it induces actions on $\mathcal{X}$ and $P$. We denote the composition of the blow-down $\mathcal{X} \to X \times \mathbb{C}$ and the projection to $X$ by $p:\mathcal{X} \to X$. Let us define the $\mathbb{Q}$-line bundle $\mathcal{L}_c$ on $X$ by $\mathcal{L}_c:=p^*L \otimes \mathcal{O}(-cP)$. 
When $V=V(\mathcal{J})$ is smooth, $P$ is a compactification of the normal bundle $N_{V/X}$. This is why we call $(\mathcal{X}, \mathcal{L}_c)$ the deformation to the normal cone. Let us denote the blow-up along $\mathcal{J}$ by $\mu: X' \to X$ and the exceptional divisor by $E$. The Seshadri constant of $L$ along $\mathcal{J}$ is defined by  
\begin{equation*}
\varepsilon(L, \mathcal{J}):= \sup \bigg\{ \ c \  \bigg| \ \mu^*L\otimes \mathcal{O}(-cE) \text{ is ample} \bigg\}. 
\end{equation*}
Then we have the following lemma so that $(\mathcal{X}, \mathcal{L}_c)$ actually defines a test configuration for any sufficiently small $c$. 
\begin{lem}[\cite{RT06}, Lemma 4.1]
For any $0<c<\varepsilon(L, \mathcal{J})$, $\mathcal{L}_c$ is a $\pi$-ample $\mathbb{Q}$-line bundle. 
\end{lem}
\end{exam}

The slop theory of \cite{RT06} was further developed by \cite{Oda09}. We will use it in the next subsection to compute the associated graded linear series of $\mathcal{T}$. Consider a flag of ideal sheaves $\mathcal{J}_0 \subseteq \mathcal{J}_1 \subseteq \cdots \subseteq \mathcal{J}_{N-1}\subseteq\mathcal{O}_X$ and fix $c\in \mathbb{Q}_{>0}$. Let us take the blow
up $\mathcal{X}$ of $X \times \mathbb{C}$ along the $\mathbb{C}^*$-invariant ideal sheaf 
\begin{equation*}
 \mathcal{J}:= \mathcal{J}_0 + t \mathcal{J}_1 + \cdot + t^{N-1}\mathcal{J}_{N-1} + (t^N) 
\end{equation*} 
and denote the exceptional divisor by $P$ and the projection map by $p: \mathcal{X} \to X$. Then $\mathcal{X}$ naturally admits a line bundle $\mathcal{L}:= p^*L\otimes\mathcal{O}(-cP)$. In his paper \cite{Oda09} Odaka derived intersection number formula of the Donaldson--Futaki invariant for this type of semi test configuration defined by flag ideals. The point is that any test configuration can be dominated by the above type of {\em semi} test configuration. 

\begin{prop}[Proposition $3.10$ of \cite{Oda09}]\label{flag}
For an arbitrary normal test configuration $\mathcal{T}$, there exist a flag of ideal sheaves $\mathcal{J}_0 \subseteq \mathcal{J}_1 \subseteq \cdots \subseteq \mathcal{J}_{N-1}\subseteq\mathcal{O}_X$ and $c\in\mathbb{Q}_{>0}$ such that $\mathcal{T}'=(\mathcal{X'}, \mathcal{L'})$ defined by the flag is a semi test configuration which dominates $\mathcal{T}$ by a morphism $f \colon \mathcal{X}' \to \mathcal{X}$ with $\mathcal{L}'=f^*\mathcal{L}$. Moreover, $F_1(\mathcal{T}')= F_1(\mathcal{T})$ holds. 
\end{prop}

%Subsection%
\subsection{The associated family of graded linear series}\label{The associated family of graded linear series}

Let us denote the $\mathbb{C}^*$-action on $(\mathcal{X}, \mathcal{L})$ by $\rho: \mathbb{C}^* \to \Aut( \mathcal{X}, \mathcal{L})$. For any $s\in H^0(X, L^{\otimes k})$, it naturally defines an invariant section $\bar{s} \in H^0(\mathcal{X}_{t\neq 0}, \mathcal{L}^{\otimes k})$ by $\bar{s}(\rho(\tau)x):=\rho(\tau)s(x)$\ $(\tau \in \mathbb{C}^*, x \in \mathcal{X}_{t\neq 0})$. 
\begin{lem}[\cite{WN10}, Lemma 6.1.]
Let $t$ be the parameter of underlying space $\mathbb{C}$. For any $\lambda \in \mathbb{Z}, t^{-\lambda}\bar{s}$ defines a meromorphic section of $\mathcal{L}^{\otimes k}$ over $\mathcal{X}$. 
\end{lem}
We then introduce the following filtration to measure the order of these meromorphic sections along the central fiber. 
\begin{dfn}
Fix a test configuration $(\mathcal{X}, \mathcal{L})$. For each $\lambda \in \mathbb{R}$, we define the subspace of $H^0(X, L^{\otimes k})$ by 
\begin{equation}
\mathcal{F}_{\lambda}H^0(X, L^{\otimes k}):= \bigg\{ \ s\in H^0(X, L^{\otimes k}) \ \bigg| \ t^{-\lceil \lambda \rceil}\bar{s} \in H^0(\mathcal{X}, \mathcal{L}) \bigg\}. 
\end{equation}
\end{dfn}
By definition $(\rho(\tau)s)(x):=\rho(\tau)s(\rho^{-1}(\tau)(x))$ so it holds $(\rho(\tau)\bar{s})(x)=\rho(\tau)\bar{s}(\rho^{-1}(\tau))(x)=\bar{s}(x)$ {\em i.e.} $\bar{s}$ is invariant under the $\mathbb{C}^*$-action. On the other hand, regarding $t$ as the section of $\mathcal{O}$ we have 
$(\rho(\tau)t)(x)=\rho(\tau)t(\rho^{-1}(\tau)x)=\rho(\tau)(\tau^{-1}t(x))=\tau^{-1}t(x)$. Therefore $t^{-\lceil \lambda \rceil}\bar{s}$ is an eigenvector of weight $\lceil \lambda \rceil$ with respect to the $\mathbb{C}^*$-action. Note that the filtration is multiplicative, {\em i.e.} 
\begin{equation*}
\mathcal{F}_{\lambda}H^0(X, L^{\otimes k}) \cdot \mathcal{F}_{\lambda'}H^0(X, L^{k'}) \subset \mathcal{F}_{\lambda+\lambda'}H^0(X, L^{k+k'})
\end{equation*}
holds for any $\lambda, \lambda' \in \mathbb{R}$ and $k, k' \geq 0$.
The relation to the weight of the action on the central fiber is given by the following proposition. 
\begin{prop}
Let us denotes the weight decomposition of the $\mathbb{C}^*$-action by $H^0(\mathcal{X}_0, \mathcal{L}_{\mathcal{X}_0})=\bigoplus_{\lambda} V_{\lambda}$. 
Then, for any $\lambda \in \mathbb{R}$ we have 
\begin{equation}\label{weight and filtration}
\dim \mathcal{F}_{\lambda}H^0(X, L^{\otimes k}) = \sum_{\lambda' \geq \lambda} \dim V_{\lambda'}. 
\end{equation}
\end{prop}
Note that every weight is actually an integer so that each side of (\ref{weight and filtration}) is unchanged if one replaces $\lambda$ to $\lceil \lambda \rceil$. 
 The fundamental fact established in \cite{PS07} is that this filtration is actually linearly bounded in the following sense. 
\begin{lem}[\cite{PS07}, Lemma 4]\label{linearly bounded}
 For any test configuration $(\mathcal{X}, \mathcal{L})$ there exists a constant $C>0$ such that for any $k \geq 1$ and $\lambda$ with $\dim V_{\lambda}>0$, 
 \begin{equation*}
 \abs{\lambda} \leq Ck
 \end{equation*}
holds. 
\end{lem}
In other words, there exists a constant $C>0$ such that 
\begin{align*}
 \mathcal{F}_{-Ck}H^0(X, L^{\otimes k}) = H^0(X, L^{\otimes k}) \ \ \ \text{and} \ \ \ 
 \mathcal{F}_{Ck}H^0(X, L^{\otimes k}) = \{0\}
\end{align*}
hold for every $k \geq 1$. 
\begin{dfn}
We set 
\begin{align*}
& \lambda_0 := \sup \big\{ \lambda \ \big| \ \mathcal{F}_{\lambda k} H^0(X, L^{\otimes k}) =H^0(X, L^{\otimes k}) \ \text{for any} \ k\geq 1 \big\} \ \ \text{and} \  \\
&\lambda_c := \inf \big\{ \lambda \ \big| \ \mathcal{F}_{\lambda k} H^0(X, L^{\otimes k}) =\{0\} \ \text{for any} \ k\geq 1 \big\}. 
\end{align*}
\end{dfn}
By Lemma \ref{linearly bounded}, $\lambda_0$ and $\lambda_c$ are both finite. Lemma \ref{linearly bounded} indicate us to consider the graded linear series 
\begin{equation}
W_{\lambda}= \bigoplus_{k=0}^{\infty} W_{\lambda, k}:=\bigoplus_{k=0}^{\infty} \mathcal{F}_{\lambda k}H^0(X, L^{\otimes k})
\end{equation}
For each $\lambda \in \mathbb{R}$. 
It was shown by \cite{Sz11} that this family has a sufficient information of original test configuration. A result of \cite{WN10} in fact gives the explicit formula for $b_0$. 

\begin{thm}[A reformulation of \cite{WN10}, Corollary 6.6]\label{b_0}
Let $(\mathcal{X}, \mathcal{L})$ be a test configuration. Then the quantity $b_0$ is obtained by the Lebesgue--Stieltjes integral of $\lambda$ with respect to $\vol(W_{\lambda})$. That is, 
\begin{equation*}
n!b_0 = -\int_{-\infty}^{\infty} \lambda d(\vol(W_{\lambda})). 
\end{equation*}
\end{thm}
Theorem \ref{b_0} actually follows from Corollary 6.6 of \cite{WN10} by change of variables in integration. Note that the concave function $G[\mathcal{T}]$ on the Okounkov body $\Delta(L)$ in \cite{WN10} is determined by the property: $G[\mathcal{T}]^{-1}([\lambda, \infty))=\Delta(W_{\lambda})$ where $\Delta(W_{\lambda}) \subset \mathbb{R}^n$ is the Okounkov body of $W_{\lambda}$ in the sense of \cite{LM08}, Definition 1.15 and that $n!$ times the Euculidian volume $\vol(\Delta(W_{\lambda}))$ equals to $\vol(W_{\lambda})$. Here however we give a self-contained proof of the above theorem. Our proof is rather simple than that of \cite{WN10} which used the method of Okounkov body. 

Set the counting function of weights as
\begin{equation}
f(\lambda)=f_k(\lambda):= \sum_{\lambda' \geq \lambda} \dim V_{\lambda'} = \dim \mathcal{F}_{\lambda}H^0(X, L^{\otimes k}). 
\end{equation}
It is easy to show that $f_k(\lambda)$ is actually left-continuous and non-increasing function. Hence the Lebesgue--Stieltjes integral makes sense and 
\begin{align*}
w(k) & := \sum_{\lambda} \lambda \dim V_{\lambda} =b_0k^{n+1}+b_1k^n+O(k^{n-1})\\
& = -\int_{-\infty}^{\infty} \lambda df(\lambda) = -\int_{-\infty}^{\infty} k\lambda df(k\lambda) 
\end{align*}
hold for any $k$. 
For any small $\varepsilon>0$ Integration by part yields 
\begin{align*}
 -\int_{-\infty}^{\infty} k\lambda df(k\lambda) 
 &= -\bigg[ k\lambda f(k\lambda) \bigg]_{\lambda_0 -\varepsilon }^{\infty} + \int_{\lambda_0-\varepsilon }^{\infty} k f(k\lambda) d\lambda. 
\end{align*}
By the definition of the volume we have 
\begin{equation*}
\limsup_{k \to \infty} \frac{f(k\lambda)}{k^n/n!} = \vol(W_{\lambda}). 
\end{equation*}
If $\vol(W_{\lambda})>0$, the limit of supremum is in fact limit for $k$ sufficiently divisible, by Theorem $4$ of \cite{KK09} (or by the proof of Theorem 3.10, Corollary 3.11, and Lemma 3.2 of \cite{DBP12}). Therefore the dominate convergence theorem concludes 
\begin{equation*}
n! b_0 = (\lambda_0 -\varepsilon)L^n + \int_{\lambda_0 -\varepsilon}^{\infty} \vol(W_{\lambda}) d\lambda. 
\end{equation*}
Thus we obtain Theorem \ref{b_0}. 

We remark that one of the advantage to consider such graded linear series is to avoid the difficulty comes from the singularity of the central fiber $\mathcal{X}_0$. On the other hand, we have to treat with the difficulty comes from the non-completeness of linear series in this setting. 

\begin{exam}
Let $(\mathcal{X}, \mathcal{L})$ be the test configuration defined by an ideal sheaf $\mathcal{J} \subseteq \mathcal{O}_X$ and $c \in \mathbb{Q}$ as in Example \ref{RT}. Then the associated $W_{\lambda}$ are computed to be: 
\begin{eqnarray*}
W_{\lambda, k} = 
\left\{
\begin{array}{ll}
H^0(X, L^{\otimes k}) 
&(\lambda \leq -c) \\
H^0(X, L^{\otimes k}\otimes \mathcal{J}^{\lceil \lambda k\rceil+ck}) 
&(-c < \lambda \leq 0) \\
\{0\} 
&(\lambda > 0) 
\end{array}
\right.
\end{eqnarray*}
for any $k$. 
As a result, we have 
\begin{equation*}
n!b_0 = -cL^n + \int_{-c}^0 \big(\mu^*L\otimes \mathcal{O}(-(\lambda + c)E)\big)^n d\lambda . 
\end{equation*}
\end{exam}

In fact thanks to Proposition \ref{flag} one can compute $W_{\lambda}$ for any test configuration. 
\begin{prop}
For any test configuration $\mathcal{T}$, there exist a flag of ideals $\mathcal{J}_0 \subseteq \mathcal{J}_1 \subseteq \cdots \subseteq \mathcal{J}_{N-1}\subseteq\mathcal{O}_X$ and $c\in\mathbb{Q}_{>0}$ such that $W_{\lambda, k}$ can be computed for any $\lambda \in \mathbb{Q}_{>0}$ and $k$ sufficiently divisible, as follows: 
\begin{equation*}
W_{\lambda, k} = H^0(X, L^{\otimes k}) \ \ \ \text{if} \ \ \ \lambda \leq -Nc, 
\end{equation*} 
\begin{equation*}
W_{\lambda, k} = \{0\} \ \ \ \text{if} \ \ \ \lambda > \frac{N(N-1)}{2}c, \ \ \ \text{and}
\end{equation*} 
\begin{equation*}
W_{\lambda, k} = H^0(X, L^{\otimes k}\otimes \mathcal{J}_{N-1}^{ck}\mathcal{J}_{N-2}^{ck}\cdots\mathcal{J}_{N-j+1}^{ck}\mathcal{J}_{N-j}^{\frac{\lceil \lambda k\rceil+(N-\frac{j(j-1)}{2}))ck}{j}}) 
%&(-(N-\frac{N(N-1)}{2})c < \lambda \leq \frac{N(N-1)}{2}c)
\end{equation*} 
if $-(N-\frac{j(j-1)}{2})c < \lambda \leq -(N-\frac{j(j+1)}{2})c$ holds for some $1\leq j\leq N$. 
%In particular, 
%\begin{equation*}
%\lambda_0=-Nc \ \ \ \text{and} \ \ \ \lambda_c= \frac{N(N-1)}{2}c. 
%\end{equation*}

%\begin{eqnarray*}
%W_{\lambda, k} = 
%\left\{
%\begin{array}{ll}
%H^0(X, L^{\otimes k}) 
%&(\lambda \leq -Nc) \\
%H^0(X, L^{\otimes k}\otimes \mathcal{J}_{N-1}^{\lceil \lambda k\rceil+Nck}) 
%&(-Nc < \lambda \leq -(N-1)c) \\
%H^0(X, L^{\otimes k}\otimes \mathcal{J}_{N-1}^{ck}\mathcal{J}_{N-2}^{\frac{\lceil \lambda k\rceil+(N-1)ck}{2}}) 
%&(-(N-1)c < \lambda \leq -(N-3)c) \\
%H^0(X, L^{\otimes k}\otimes \mathcal{J}_{N-1}^{ck}\mathcal{J}_{N-2}^{ck}\mathcal{J}_{N-3}^{\frac{\lceil \lambda k\rceil+(N-3)ck}{3}}) 
%&(-(N-3)c < \lambda \leq -(N-6)c) \\
%\cdot \\
%\cdot \\
%\cdot \\
%H^0(X, L^{\otimes k}\otimes \mathcal{J}_{N-1}^{ck}\mathcal{J}_{N-2}^{ck}\cdots\mathcal{J}_{1}^{ck}\mathcal{J}_{0}^{\frac{\lceil \lambda k\rceil+(N-\frac{N(N-1)}{2})ck}{N}}) 
%&(-(N-\frac{N(N-1)}{2})c < \lambda \leq \frac{N(N-1)}{2}c) \\ 
%\{0\} 
%&(\lambda > \frac{N(N-1)}{2}c) 
%\end{array}
%\right.
%\end{eqnarray*}
%for any $k$ sufficiently divisible. 

\end{prop}
\begin{proof}
Let us take $f \colon \mathcal{X}' \to \mathcal{X}$ as Proposition \ref{flag}. Note that $f$ is isomorphic except on the locus contained in $\mathcal{X}_0$, whose codimension is greater than $2$ in $\mathcal{X}$. Then it is easy to see that $W_{\lambda, k}$ for $\mathcal{X}$ is naturally isomorphic to that for $\mathcal{X}'$. Therefore we may assume that $\mathcal{X}$ is defined by a flag $\mathcal{J}_0 \subseteq \mathcal{J}_1 \subseteq \cdots \subseteq \mathcal{J}_{N-1}\subseteq\mathcal{O}_X$, without loss of generality. Then we have the decomposition 
\begin{align*}
H^0(\mathcal{X}, \mathcal{L}^{\otimes k})
= H^0(X\times \mathbb{C}, p_1^*L^{\otimes k}\otimes (\mathcal{J}_0 +t\mathcal{J}_1+ \cdots +t^{N-1} \mathcal{J}_{N-1} +(t^N))^{ck}) \\
= \bigg(\bigoplus_{i_0+i_1+\cdots +i_N=ck} t^{i_1+2i_2+\cdots +Ni_N} H^0(X, L^{\otimes k}\otimes \mathcal{J}_0^{i_0}\mathcal{J}_1^{i_1}\cdots \mathcal{J}_{N-1}^{i_{N-1}}) \bigg) \oplus t^{Nck} \mathbb{C}[t]H^0(X, L^{\otimes k}).  
\end{align*}
To compute $\lceil \lambda k\rceil$-component of $H^0(\mathcal{X}_0, \mathcal{L}_0^{\otimes k})$, we should solve 

\begin{eqnarray*}
\left\{
\begin{array}{ll}
& i_1+2i_2+\cdots +Ni_N=- \lceil \lambda k\rceil \\
& i_0+i_1+\cdots +i_N=ck
\end{array}
\right.  
\end{eqnarray*}
\begin{align*}
\Leftrightarrow
Ni_0 + (N-1)i_1+\cdots +i_{N-1}=\lceil \lambda k\rceil +Nck. 
\end{align*}
The above equation for $(i_0, \dots, i_{N-1})$ has many solutions but if two solutions satisfy $(i_0, \dots, i_{N-1}) \prec (j_0, \dots, j_{N-1})$ in the lexicographic order 
\begin{equation*}
H^0(X, L^{\otimes k}\otimes \mathcal{J}_0^{i_0}\mathcal{J}_1^{i_1}\cdots \mathcal{J}_{N-1}^{i_{N-1}}) \subseteq H^0(X, L^{\otimes k}\otimes \mathcal{J}_0^{j_0}\mathcal{J}_1^{j_1}\cdots \mathcal{J}_{N-1}^{j_{N-1}})
\end{equation*}
so that $W_{\lambda, k} = H^0(X, L^{\otimes k}\otimes \mathcal{J}_0^{i_0}\mathcal{J}_1^{i_1}\cdots \mathcal{J}_{N-1}^{i_{N-1}})$ holds for the maximal solution $(i_0, \dots, i_{N-1})$. 

\end{proof}

If one take a resolution $\mu \colon X' \to X$ of $\mathcal{J}_0, \dots , \mathcal{J}_{N-1}$ and divisors $E_j$ $(1\leq i\leq N-1)$ on $X'$ such that $\mathcal{J}\mathcal{O}_{X'} = \mathcal{O}_{X'}(-E_i)$ and $E_i \geq E_j$ $(i \leq j)$ hold, it holds that 

\begin{align*}
W_{\lambda, k} \simeq H^0(X', \mu^*L^{\otimes k} \otimes \mathcal{O}(-kE_{\lambda})) 
\end{align*}
where 
\begin{equation*}
E_{\lambda}:= cE_{N-1}+\cdots +cE_{N-j+1}+ \frac{\lambda + (N-\frac{j(j-1)}{2})c}{j}E_{N-j}. 
\end{equation*}
The above computation is not so practical but here we obtain the following observation so that we may apply Theorem \ref{main1} to $W_{\lambda}$. %(See also the remark below Theorem \ref{main1}.)
\begin{cor}\label{birational}
If $\lambda < \lambda_c$, the natural map $X \dashrightarrow \mathbb{P}W_{\lambda, k}^*$ is birational to the image for any $k$ sufficiently divisible. 
\end{cor}
\begin{proof}
Note that $\lambda_c \leq \frac{N(N-1)}{2}c$. 
Let us first see that $\mu^*L \otimes \mathcal{O}(-E_{\lambda})$ is big for any $\lambda < \lambda_c$. It is enough to consider the case where there exists a $j$ such that $E_{\lambda}$ is big for $\lambda< -(N-\frac{j(j+1)}{2})c$ but not so for $\lambda = -(N-\frac{j(j+1)}{2})c$. Since $E_{N-j-1} \geq E_{N-j}$ this actually implies $E_{\lambda}$ is not pseudo-effective for $\lambda > -(N-\frac{j(j+1)}{2})c$. Let us now fix a rational number $\lambda'$ such that $\lambda < \lambda' < \lambda_c$ holds. Then $W_{\lambda, k}$ contains $W_{\lambda', k} \simeq H^0(X', \mu^*L^{\otimes k} \otimes \mathcal{O}(-kE_{\lambda'}))$ for sufficiently divisible $k$. Since $\mu^*L \otimes \mathcal{O}(-E_{\lambda'})$ is big, this concludes the corollary. 
\end{proof}

%Some consequences on the weak geodesic ray% 
\section{Study of the weak geodesic ray}\label{Study of the weak geodesic ray} 

In this section we apply Theorem \ref{main1} to each $W_{\lambda}$ constructed from the test configuration to study the associated weak geodesic ray. 

\subsection{Construction of weak geodesic}\label{Construction of weak geodesic}

One of the guiding principles to the existence problem of constant scalar curvature K\"{a}hler metric is to study the Riemannian geometry on the space of K\"{a}hler metrics in the first Chern class of $L$. A result of Phong and Sturm (in \cite{PS07}) gives a milestone in this direction. They showed that a test configuration canonically defines a weak geodesic ray emanating from any fixed point $\varphi$ in the space of K\"{a}hler metrics. This builds a bridge between the algebraic definition of K-stability and the analytic stage where the cscK metric lives. 
Later it was shown by \cite{RWN11} that one can also define the same weak geodesic via the associated family of graded linear series $\{W_{\lambda}\}$. Let us now recall their construction. Throughout this subsection we fix a smooth strictly psh weight $\varphi$. It will be shown that $\varphi$ and the family of graded linear series $\{ W_{\lambda} \}$ canonically define the weak geodesic emanating from $\varphi$. 

Recall that a family of psh weights $\psi_t$ $(a <t<b)$ is called {\em weak geodesic} if $\Psi(x, \tau):=\psi_{-\log\abs{\tau}}(x)$ $(\tau \in \mathbb{C}, e^{-b}<\abs{\tau} <e^{-a})$ is plurisubharmonic and satisfies the Monge--Amp\`{e}re equation 
\begin{equation*}
\MA(\Psi)=0. 
\end{equation*}
Here we consider $\Psi(x, \tau)$ as the function of $(n+1)$-variables and the Monge-Amp\`{e}re operator is defined in subsection $2.1$. When each $dd^c \varphi_t$ is a smooth K\"{a}hler metric, there is the canonical Riemannian metric which is defined for a tangent vector $u$ at $\varphi_t$ by 
\begin{equation*}
\norm{u}^2:=\int_X u^2 \frac{\MA(\varphi_t)}{n!}. 
\end{equation*}
By \cite{Sem92}, it is known that $\MA(\Psi)=0$ if and only if the geodesic curvature for this metric is zero. 

First note that given test configuration $(\mathcal{X}, \mathcal{L})$, the associated family $\{ W_{\lambda} \}$ defines the family of equilibrium weight $P_{W_{\lambda}}\varphi$. Let us from now write as 
\begin{equation*}
\psi_{\lambda}:=P_{W_{\lambda}}\varphi. 
\end{equation*}
The first easy observation is that  $\psi_{\lambda}$ is decreasing with respect to $\lambda$. As a consequence of the Bergman approximation argument by Demailly and Lemma \ref{linearly bounded}, we have 
\begin{align*}
\psi_{\lambda}= \varphi \ \ \text{if}\ \lambda < \lambda_0 \ \ \ \  \text{and} \ \ \ \ \lambda_c=\inf \{ \ \lambda \ | \ \psi_{\lambda} = -\infty \ \}. 
\end{align*}
Further by the multiplicativity of $\mathcal{F}_{\lambda} H^0(X, L^{\otimes k})$ one can see that $\psi_{\lambda}$ is concave with respect to $\lambda$. The main result of \cite{RWN11} states that the Legendre transformation of $\psi_{\lambda}$ defines a weak geodesic ray.  

\begin{thm}[\cite{RWN11}, Theorem 1.1 and Theorem 1.2, Theorem 9.2.]
Set the Legendre transformation of $\psi_{\lambda}$ by  
\begin{equation} 
\varphi_t := {\sup}^*\big\{ \psi_{\lambda} + t\lambda \ \big| \ \lambda \in \mathbb{R} \big\} \ \ \ \text{for} \ \ t\in[0, +\infty). 
\end{equation} 
Then $\varphi_t$ defines a weak geodesic emanating from $\varphi$. Moreover, $\varphi_t-F_0$ coincide with the weak geodesic ray constructed in \cite{PS07}. 
\end{thm}

It is immediate to show that $\varphi_t$ is a bounded psh weight emanating from $\varphi$ and that it is convex with respect to $t$. 
The geodecity is derived from the maximality of $P_{W_{\lambda}}\varphi$, that is, 
\begin{equation}\label{maximality}
\psi_{\lambda} =\varphi \ \ \text{a.e. with respect to} \ \MA(\psi_{\lambda}). 
\end{equation}
And one of the technical point in \cite{RWN11} is to show (\ref{maximality}). Such property is caused by the fact that $\psi_{\lambda}$ is defined as the upper envelops of sufficiently many algebraic weights. 

Note that the inverse Legendre transform maps $\varphi_t$ to $\psi_{\lambda}$ by 
\begin{equation}\label{Legendre}
\psi_{\lambda}= \inf_t \big\{ \varphi_t -t\lambda \big\} 
\end{equation}
which holds on almost every point on $X$. Therefore the two curves have the equivalent information. 
Fix $t\in[0, \infty)$. By the convexity of $\varphi_t$ in $t$, the right derivative $\dot{\varphi_t}(x)$ is defined for every $x \in X$. We identify this right derivative with the tangent vector of the weak geodesic. Moreover, the gradient map relation 
\begin{equation}\label{gradient map relation}
-\psi_{\lambda}(x) + \varphi_t(x) = t \lambda
\end{equation}
holds almost everywhere if one set $\lambda := \dot{\varphi}_t(x)$.

\subsection{Proof of Theorem \ref{main2}}\label{proof of theorem main2}
Now we prove Theorem \ref{main2}. It was shown in \cite{WN10} that the push-forward of the Lebesgue measure by the concave function $G[\mathcal{T}]$ on the Okounkov body $\Delta(L)$ gives the weak limit. That is,  
\begin{equation}\label{WN}
\lim_{k \to \infty} \frac{n!}{k^n}\sum_{\lambda}\delta_{\frac{k}{\lambda}} \dim V_{\lambda} = n!G[\mathcal{T}]_*(d\lambda|_{\Delta(L)}). 
\end{equation}
Recall that $G[\mathcal{T}]$ is characterized by its property: $G[\mathcal{T}]^{-1}([\lambda, \infty)) = \Delta(W_{\lambda})$ where $\Delta(W_{\lambda}) \subseteq \mathbb{R}^n$ is the Okounkov body of $W_{\lambda}$ in the sense of \cite{LM08}, Definition 1.15 and $n!$ times the Euculidian volume $\vol(\Delta(W_{\lambda}))$ gives $\vol(W_{\lambda})$.
Therefore it is easy to observe that the right hand side of (\ref{WN}) equals to $-d(\vol(W_{\lambda}))$. Then, by Theorem \ref{main1} with Corollary \ref{birational}, we may reduce the proof of Theorem \ref{main2} to show  
\begin{equation}\label{target}
 -d\int_X\MA(\psi_{\lambda})=(\dot{\varphi}_t)_*\MA(\varphi_t). 
\end{equation}
Here we used the assumption $\mathcal{X}$ is normal, in order to apply Corollary \ref{birational}. 
By the main result of \cite{PS10} $\varphi_t$ has the $C^{1, \alpha}$-regularity so that we can apply Proposition $2.2$ of \cite{Bern09}. Then it can be seen that the right hand side of (\ref{target}) is independent of $t$ and the proof is reduced to the case $t=0$. Then by basic measure theory we conclude Theorem \ref{main2} if for any $\lambda \in \mathbb{R}$
\begin{equation}\label{point}
 \int_X \MA(\psi_{\lambda}) = \int_{\{\dot{\varphi}_0 \geq \lambda\}} \MA(\varphi)
\end{equation}
holds. Or it is sufficient to show 
\begin{equation}\label{point'}
\int_{\{\dot{\varphi}_0 > \lambda\}} \MA(\varphi) \leq \int_X \MA(\psi_{\lambda}) \leq \int_{\{\dot{\varphi}_0 \geq \lambda\}} \MA(\varphi) 
\end{equation}
for any $\lambda \in \mathbb{R}$. 
The following lemma is directly deduced from the definition of $\varphi_t$. 
\begin{lem}
For almost every point in $X$, $\dot{\varphi}_0 \geq \lambda$ holds if and only if $\psi_{\lambda}=\varphi$. In particular 
\begin{equation*}
\int_{\{\dot{\varphi}_0 \geq \lambda\}}\MA(\varphi)=\int_{\{\psi_{\lambda}=\varphi\}}\MA(\varphi). 
\end{equation*}
holds. 
\end{lem}
\begin{proof}
Let $x$ be a point of $X$. If $\psi_{\lambda}(x)=\varphi(x)$, then 
\begin{equation*}
\dot{\varphi_0}(x):= \inf_t \frac{\varphi_t(x)- \varphi(x)}{t} 
\geq \frac{\psi_{\lambda}(x) + t\lambda -\varphi(x)}{t}
\geq \lambda. 
\end{equation*}
On the other hand, by the Legendre relation (\ref{Legendre}), $\dot{\varphi_0}(x) \geq \lambda$ yields
\begin{align*}
\psi_{\lambda}(x) & = \inf_t \big\{ \varphi_t(x) -t\lambda \big\} \\
& \geq \inf_t \big\{ t\dot{\varphi_0}(x) + \varphi(x) -t\lambda \big\}
\geq \varphi(x)
\end{align*}
for almost every $x\in X$. 
 
\end{proof}

In the case of Example \ref{RT}, the result of \cite{Berm07} yileds much stronger conclusion that $\psi_{\lambda}$ has $C^{1, 1}$-regularity on the bounded locus and 
\begin{equation*}
\MA(\psi_{\lambda}) = \mathrm{1}_{\{\psi_{\lambda}=\varphi\}}\MA(\varphi)
\end{equation*}
holds. Here, however, we give a proof of  (\ref{point'}) without the regularity of $\psi_{\lambda}$. Note that the set $\{ \dot{\varphi}_0 > \lambda \}$ is open (thanks to the regularity result of \cite{PS10}) and contained in $\{\psi_{\lambda} = \varphi\}$. It was shown by \cite{BEGZ10} that the Monge--Amp\`{e}re product is local in the plurifine topology. Therefore we have 
\begin{equation*}
\int_{\{\dot{\varphi}_0 > \lambda \}} \MA(\psi_{\lambda}) = \int_{\{\dot{\varphi}_0 > \lambda \}} \MA(\varphi). 
\end{equation*}
Then we obtain the one side inequality of (\ref{point'}), 
\begin{equation*}
\int_X \MA(\psi_{\lambda}) \geq \int_{\{\dot{\varphi}_0 > \lambda\}} \MA(\varphi). 
\end{equation*}

Let us take any $\varepsilon>0$ to prove the converse inequality. Thanks to the maximality (\ref{maximality}) we have 
\begin{align*}
\int_X \MA(\psi_{\lambda}) 
=\int_{\{\psi_{\lambda}>\varphi-\varepsilon\}} \MA(\psi_{\lambda})
= \int_{\{\psi_{\lambda}>\varphi-\varepsilon\}} \MA(\max{\{\psi_{\lambda}, \varphi-\varepsilon\}}). 
\end{align*}
Note that the set $\{\psi_{\lambda}>\varphi-\varepsilon\}$ is pluri-open. 
The right hand side equals to 
\begin{align*}
L^n - \int_{\{\psi_{\lambda} \leq \varphi-\varepsilon\}} \MA(\max{\{\psi_{\lambda}, \varphi-\varepsilon\}})
\end{align*}
by Theorem \ref{comparison theorem}. Therefore we obtain 
\begin{align*}
\int_X \MA(\psi_{\lambda}) 
&\leq L^n - \int_{\{\psi_{\lambda} < \varphi-\varepsilon\}} \MA(\max{\{\psi_{\lambda}, \varphi-\varepsilon\}}) \\
& = L^n - \int_{\{\psi_{\lambda} < \varphi-\varepsilon\}} \MA(\varphi). 
\end{align*}
If $\varepsilon>0$ tends to $0$ then the set $\{\psi_{\lambda} < \varphi-\varepsilon\}$ converges to $\{\dot{\varphi}_0 <\lambda \}$ hence 

\begin{equation*}
\int_X \MA(\psi_{\lambda}) \leq \int_{\{\dot{\varphi}_0 \geq \lambda\}} \MA(\varphi). 
\end{equation*}

This ends the proof. 

%Set $f(\lambda):=\int_{\{\dot{\varphi}_0 \geq \lambda\}}\MA(\varphi)$ and $g(\lambda):=\int_X \MA(P_{W_{\lambda}}\varphi)=\vol(W_{\lambda})$. We showed $f-g\leq 0$. On the other hand it is well known that 
%\begin{equation*}
%f(\lambda)-g(\lambda) = 0 \ \ \ \text{if} \ \lambda \ \text{is sufficiently small or big }
%\end{equation*}
%and
%\begin{equation*}
%\int_{-\infty}^{\infty} \lambda d(f(\lambda)-g(\lambda)) =0. 
%\end{equation*}
%The second identity is the gradient formula for the Aubin--Mabuchi energy. For the detail, see the remark below in Theorem \ref{main3'}. Then the integration by part easily deduce $f=g$. This ends the proof. 

\subsection{Norms on the weak geodesic ray}\label{Norms on the weak geodesic ray}

We conclude this paper by discussing some consequences of Theorem \ref{main2}, which are concerned with the $p$-norm of test configuration. 
\begin{dfn}
Fix any test configuration $(\mathcal{X}, \mathcal{L})$ of a polarized manifold $L$. Let $H^0(\mathcal{X}_0, \mathcal{L}_0^{\otimes k})=\bigoplus_{\lambda}V_{\lambda}$ be the weight decomposition of the induced $\mathbb{C^*}$-action. Define the trace-free part of each eigenvalue as 
\begin{equation*}
\bar{\lambda}:= \lambda - \frac{1}{N_k}\sum_{\lambda}\lambda \dim V_{\lambda}
\end{equation*}
and introduce the $p$-norms ($p\in\mathbb{Z}_{\geq0}$) of the test configuration by 
\begin{equation*}
Q_p := \lim_{k \to \infty} \frac{1}{k^n}\sum_{\lambda} \bigg(\frac{\lambda}{k}\bigg)^p\dim V_{\lambda} 
\end{equation*}
and 
\begin{equation*}
N_p := \lim_{k \to \infty} \frac{1}{k^n}\sum_{\lambda} \bigg(\frac{\bar{\lambda}}{k}\bigg)^p\dim V_{\lambda}. 
\end{equation*}
Especially in the case $p=2$ we denote $Q_2$ and $N_2$ by $Q$ and $\norm{\mathcal{T}}^2=\norm{\mathcal{T}}_2^2$. Note that the limits exist since the summations in the right-hand side can be thought as the appropriate Hilbert polynomial. 
\end{dfn}
It is easy to see that $Q_1=b_0$, $N_1=0$, ${N_2}=Q_2-\frac{{b_0}^2}{a_0}$, and 
\begin{equation*}
\frac{1}{N_k}\sum_{\lambda} \frac{\lambda}{k}  \to F_0=\frac{b_0}{a_0}. 
\end{equation*}
These norms are introduced by \cite{Don05} and played the important role in their result for the lower bound of the Calabi functional. 
We can obtain the geometric meanings of these norms in word of weak geodesic ray. 
\begin{thm}\label{main3'}
Let $(\mathcal{X}, \mathcal{L})$ be a test configuration and $\varphi_t$ be the weak geodesic associated to $(\mathcal{X}, \mathcal{L})$. Then we have 
\begin{equation*}
Q_p = \int_X (\dot{\varphi_t})^p \frac{\MA(\varphi_t)}{n!} 
\end{equation*}
and 
\begin{equation*}
N_p = \int_X \bigg(\dot{\varphi_t}- F_0\bigg)^p \frac{\MA(\varphi_t)}{n!}. 
\end{equation*}
\end{thm}

\begin{proof}
By the same argument in the proof of Theorem \ref{b_0}, we obtain 
\begin{equation*}
n! Q_p = -\int_{-\infty}^{\infty} \lambda^p d\vol(W_{\lambda}). 
\end{equation*}
This can be also obtained from the result of \cite{WN10} if one note the volume characterization of the concave function $G[\mathcal{T}]$ in \cite{WN10}. 
Taking the $p$-th moment of the two measures in Theorem \ref{main2}, we deduce the claim. The formulas for $N_p$ can be proved in the same way.  
\end{proof}
Let us examine Theorem \ref{main3'}. in the case $p=0$ it only states that $n!a_0=\int_X \MA(\varphi_t)$ and this can be easily seen from the definition of the Bedford--Taylor's Monge--Amp\`{e}re product. The case $p=1$ yields 
\begin{equation*}
n!b_0 = \int_X \dot{\varphi_t} \MA(\varphi_t). 
\end{equation*}
In other words, the Aubin--Mabuchi energy functional along the weak geodesic is given by 
\begin{equation*}
\mathcal{E}(\varphi_t, \varphi) := \int_0^1 dt \int_X \dot{\varphi_t}\MA(\varphi_t)=n!b_0t. 
\end{equation*}
(For the definition of the Aubin--Mabuchi energy of a singular Hermitian metric, see \cite{BEGZ10}.) This is a well-known result to the experts. For example, the proof of \cite{Berm12} in the Fano case works exactly the same way to yield that along the weak geodesic $b_0$ gives the gradient of the Aubin--Mabuchi energy. We have reproved it in the viewpoint of the associated family of graded linear series. It is conjectured that the gradient of the K-energy at infinity corresponds to the (minus of) Donaldson--Futaki invariant. This gives the variational approach to the existence problem. 

The most interesting case is $p=2$ which yields a part of Theorem \ref{main3} and this might be a new result. In particular, we obtain the following. 
\begin{cor}\label{trivial} 
For any test configuration, the norm $\norm{\mathcal{T}}$ is zero if and only if the associated weak geodesic ray $\varphi_t$ is $\varphi +F_0 t$. 
\end{cor} 

Only the case where the exponent $p$ is even was treated in \cite{Don05} to assure the positivity of the norm but now we may define the positive norm for odd $p$ integrating the function $\abs{\lambda}^p$, in place of $\lambda^p$, by each measure. In particular we can see that the limit 
\begin{equation*}
\norm{\mathcal{T}}_p^p := \lim_{k \to \infty} \frac{1}{k^n}\sum_{\lambda} \bigg(\frac{\abs{\bar{\lambda}}}{k}\bigg)^p\dim V_{\lambda}, 
\end{equation*}
which can not necessarily be described by a Hilbert polynomial, exists and coincide with the $L^p$ norm of the tangent vector. Thus Theorem \ref{main3} was proved. Letting $p \to +\infty$, we obtain 
\begin{equation}
\norm{\mathcal{T}}_{\infty}:= \lim_{p \to \infty} \norm{\mathcal{T}}_p = \sup_X \abs{\dot{\varphi}_t-F_0}.  
\end{equation} 
In particular the right hand side is independent of $t$ and $\varphi$. 

Let us remark some relation with \cite{Don05} and prove Theorem \ref{main4}. Let us denote the scalar curvature of the K\"{a}hler metric $dd^c\varphi$ by $S_{\varphi}$ and denote its mean value by $\hat{S}$. The main result of \cite{Don05} states that 
\begin{equation*}
(Q_p)^{\frac{1}{p}}\norm{S_{\varphi}}_{L^q} \geq b_1
\end{equation*}
and 
\begin{equation}\label{lower bound of the Calabi functional}
\norm{\mathcal{T}}_p\cdot\norm{S_{\varphi}-\hat{S}}_{L^q} \geq F_1 
\end{equation}
hold for any even $p$ and the conjugate $q$ which satisfies $1/p+1/q=1$. 
As a result one can see that the existence of constant scalar curvature K\"{a}hler metric implies K-semistability. In view of (\ref{lower bound of the Calabi functional}), \cite{Sz11} suggested the stronger notion of K-stability which implies 
\begin{equation}\label{strong K-stability}
F_1 \leq -\delta \norm{\mathcal{T}}
\end{equation}
for some uniform constant $\delta>0$. One of the motivation of this condition is that one has to consider some limit of test configurations to assure the existence of constant scalar curvature K\"{a}hler metric. 
The above condition also excludes the pathological example raised in \cite{LX11}. 
Corollary \ref{trivial} supports the validity of \cite{Sz11}'s suggestion since the gradient of the K-energy along the trivial ray $\varphi + F_0 t$ is zero. 

Let us give an energy theoretic explanation for (\ref{lower bound of the Calabi functional}). Thanks to Theorem \ref{main3}, we can apply the H\"{o}lder  inequality to obtain 
\begin{equation}\label{Holder}
\bigg(\int_X \abs{\dot{\varphi_0}-F_0}^p \frac{\MA(\varphi)}{n!}\bigg)^{\frac{1}{p}} \bigg(\int_X \abs{S_{\varphi}-\hat{S}}^q \frac{\MA(\varphi)}{n!}\bigg)^{\frac{1}{q}}  
\geq \int_X (\dot{\varphi}_0-F_0)(S_{\varphi}-\hat{S}) \frac{\MA(\varphi)}{n!} 
\end{equation}
for any pair $(p, q)$ with $1/p+1/q=1$. 
Then the right hand side is minus of the gradient of K-energy along the weak geodesic ray. The definition of the gradient for singular $\varphi_t$ is not so clear but if it was well-defined, it should be increasing with respect to $t$. Moreover the limit should be smaller just as much as the multiplicity of the central fiber than minus of the Donaldson--Futaki invariant. (See also \cite{PT06-1}, \cite{PT06-2} and \cite{PRS08}.) Assuming these points we have 
\begin{equation}\label{convexity}
\int_X (\dot{\varphi}_0-F_0)(S_{\varphi}-\hat{S}) \frac{\MA(\varphi)}{n!} 
\geq F_1. 
\end{equation}
Notice that (\ref{convexity}) implies (\ref{lower bound of the Calabi functional}) for {\em any} $1 \leq p\leq +\infty$.  
One of the proof of (\ref{convexity}) following the above line will be given in the preprint \cite{BHWN12}. 
In fact in the Fano case, we may replace the K-energy to the Ding functional to obtain the corresponding result. Convexity of the Ding functional along any weak geodesic ray was established in \cite{Bern11} and the relation between the gradient of the Ding functional and $F_1$ was shown in \cite{Berm12}. As a corollary of these results  we obtain 
\begin{equation}
\int_X (\dot{\varphi}_0-F_0)\big(e^{-\varphi} - \frac{\MA(\varphi)}{n!}\big) 
\geq F_1 
\end{equation}
with some appropriate normalization for $\varphi$. and then (\ref{Holder}) yields 
\begin{equation}\label{lower bound of the Calabi functional 2}
\norm{\mathcal{T}}_p\norm{\frac{n!e^{-\varphi}}{\MA(\varphi)}-1}_{L^q}
\geq F_1 
\end{equation}
for any $1\leq p \leq +\infty$. 
This can be seen as the analogue of the Donaldson's result in the Fano case. 

Finally we remark that the strong K-stability condition (\ref{strong K-stability}) follows from the analytic condition: 
\begin{equation}
\int_X (\dot{\varphi}_0-F_0)(S_{\varphi_t}-\hat{S}) \frac{\MA(\varphi_t)}{n!} \leq -\delta \norm{\dot{\varphi}_0-F_0}, 
\end{equation}
in case $S_{\varphi_t}$ is well-defined. It is interesting to ask whether this condition implies the properness of the K-energy. \\
 
%Acknowledgement%
{\bf Acknowledgments.}The author would like to express his gratitude 
to his advisor Professor Shigeharu Takayama for his warm encouragements, suggestions and reading the drafts. 
The author also would like to thank Professor S\'{e}bastien Boucksom 
for his indicating the relation between the authors preprint \cite{His12} and the paper of Julius Ross and David Witt-Nystr\"{o}m. The author wishes to thank Doctor David Witt-Nystr\"{o}m, Professor Robert Berman and Professor Bo Berndtsson for stimulating discussion and helpful comments during the author's stay in Gothenburg. In particular, the formulation of Theorem \ref{main2} and Theorem \ref{main3} are due to the suggestions of Doctor David Witt-Nystr\"{o}m. The proof of Theorem \ref{main2} is also indebted to his many helpful comments. Our work was carried out at several institutions including the Charmers University of Technology, the Gothenburg university, and the University of Tokyo. We gratefully acknowledge their support. The author is supported by JSPS Research Fellowships for Young Scientists (22-6742). 
 
%References%


\begin{thebibliography}{99}

\bibitem[AB84]{AB84}M. F. Atiyah and R. Bott: 
  \newblock {\em The moment map and equivariant cohomology.}
  \newblock Topology 23 (1984), no. \textbf{1}, 1--28.

\bibitem[BT76]{BT76}E. Bedford and A. Taylor: 
   \newblock {\em The Dirichlet problem for a complex Monge--Amp\`{e}re equation.}
   \newblock  Invent. Math. \textbf{37} (1976), no. 1, 1--44.

\bibitem[Berm07]{Berm07}R. J. Berman:
                     \newblock {\em Bergman kernels and equilibrium measures for ample line bundles.}
                     \newblock arXiv:0704.1640. 

\bibitem[Berm09]{Berm09}R. J. Berman:
   \newblock {\em Bergman kernels and equilibrium measures for line bundles over projective manifolds.}
   \newblock  Amer. J. Math. \textbf{131} (2009), no. 5, 1485--1524. 

\bibitem[BBGZ09]{BBGZ09}R. J. Berman, S. Boucksom, V. Guedj, and A. Zeriahi: 
   \newblock{A variational approach to complex Monge--Amp\`{e}re equations.}
   \newblock arXiv:0907.4490. 

\bibitem[Berm12]{Berm12}R. J. Berman: 
   \newblock K-polystability of Q-Fano varieties admitting Kahler--Einstein metrics. 
   \newblock arXiv:1205.6214. 

\bibitem[BHWN12]{BHWN12}R. J. Berman, T. Hisamoto, D. Witt Nystr\"{o}m: 
   \newblock {\em Calabi type functionals, test configurations and
geodesic rays.} 
   \newblock In preparation.  
   
\bibitem[Bern09]{Bern09}B. Berndtsson: 
   \newblock {\em Probability measures related to geodesics in the space of K\"{a}hler metrics.}
   \newblock arXiv:0907.1806. 
   
\bibitem[Bern11]{Bern11}
    \newblock {\em A Brunn-Minkowski type inequality for Fano manifolds and the Bando--Mabuchi uniqueness theorem}   
    \newblock arXiv:1103.0923.    

%\bibitem[BD09]{BD09}R. Berman, J. P. Demailly: 
%   \newblock {\em Regularity of plurisubharmonic upper envelopes in big %cohomology classes. }
%   \newblock  Preprint (2009) arXiv:0905.1246. 

\bibitem[DBP12]{DBP12}L. Di Biagio and G. Pacienza: 
   \newblock Restricted volumes of effective divisors. 
   \newblock arXiv:1207.1204. 

\bibitem[Bou02]{Bou02}S. Boucksom: 
   \newblock  {\em On the volume of a line bundle.}
   \newblock  Internat. J. Math. \textbf{13} (2002), no. 10, 1043--1063. 

\bibitem[BEGZ10]{BEGZ10}S. Boucksom, P. Eyssidieux, V. Guedj, and A. Zeriahi: 
   \newblock {\em Monge--Amp\`{e}re equations in big cohomology classes.}
   \newblock   Acta Math. \textbf{205} (2010), no. 2, 199--262. 
 
\bibitem[Don02]{Don02}S. K. Donaldson: 
   \newblock {\em Scalar curvature and stability of toric varieties. }
   \newblock  J. Differential Geom. \textbf{62} (2002), no. 2, 289--349.  
 
\bibitem[Don05]{Don05}S. K. Donaldson: 
   \newblock {\em Lower bounds on the Calabi functional.}
   \newblock  J. Differential Geom. \textbf{70} (2005), no. 3, 453--472.
 
%\bibitem[His11]{His11}T. Hisamoto: 
%   \newblock {\em Restricted Bergman kernel asymptotics. }
%   \newblock  Preprint (2011), arXiv:1201.4233. Trans. Amer. Math. Soc., to appear. 

\bibitem[Fuj94]{Fuj94}T. Fujita:  
    \newblock  Approximating Zariski decomposition of big line bundles.
    \newblock  Kodai Math. J. \textbf{17}, no. 1, 1--3  (1994)

\bibitem[His12]{His12}T. Hisamoto: 
   \newblock {\em On the volume of graded linear series and Monge--Amp\`{e}re mass. } 
   \newblock Preprint (2012), arXiv:1201.4236.

\bibitem[Jow10]{Jow10}S. Y. Jow: 
    \newblock Okounkov bodies and restricted volumes along very general curves. 
    \newblock Adv. Math. \textbf{223} (2010), no. 4, 1356--1371. 

\bibitem[KK09]{KK09}K. Kaveh and A. G. Khovanskii:  
   \newblock Newton convex bodies, semigroups of integral points, graded algebras and intersection theory. 
   \newblock arXiv:0904.3350. To appear in Ann. of Math. 

\bibitem[LM08]{LM08}R. Lazarsfeld and M. Musta\c{t}\u{a}: 
   \newblock  {\em Convex Bodies Associated to linear series.}
   \newblock  Ann. Sci. \'{E}c. Norm. Sup\`{e}r. (4) \textbf{42} (2009), no. 5, 783--835.   

\bibitem[LX11]{LX11}C. Li and C. Xu: 
   \newblock  {\em Special test configurations and $K$-stability of Fano varieties.}
   \newblock arXiv:1111.5398.

\bibitem[Mab08]{Mab08}T. Mabuchi: 
   \newblock K-stability of constant scalar curvature polarization
   \newblock arXiv:0812.4093. 
   
\bibitem[Mab09]{Mab09}T. Mabuchi: 
   \newblock A stronger concept of K-stability
   \newblock arXiv:0910.4617

%\bibitem[PS06]{PS06}D. H. Phong, J. Sturm:
%   \newblock {\em The Monge--Amp\`{e}re operator and geodesics in the space of K\UTF{00E4}hler potentials. }
%   \newblock  Invent. Math. \textbf{166} (2006), no. 1, 125--149.

\bibitem[Oda09]{Oda09}Y. Odaka: 
   \newblock {\em A generalization of Ross--Thomas' slope theory.}
   \newblock Preprint (2009), arXiv:1010.3597. Osaka J. Math., to appear. 
   
%\bibitem[Oda11]{Oda11}Y. Odaka: 
%   \newblock {\em The Calabi conjecture and K-stability.}
%   \newblock Int. Math. Res. Notices (2011). 

\bibitem[PT06-1]{PT06-1}S. T. Paul and G. Tian: 
     \newblock {\em CM Stability and the Generalized Futaki Invariant I. }
     \newblock arXiv: 0605278. 

\bibitem[PT06-2]{PT06-2}S. T. Paul and G. Tian: 
     \newblock {\em CM Stability and the Generalized Futaki Invariant II. }
     \newblock Ast\'{e}risque No. \textbf{328} (2009), 339--354.

\bibitem[PRS08]{PRS08}D. H. Phong, J. Ross, and J. Sturm: 
   \newblock {\em Deligne pairings and the Knudsen-Mumford expansion.}
    \newblock J. Differential Geom. 78 (2008), no. \textbf{3}, 475--496.
    
\bibitem[PS07]{PS07}D. H. Phong and J. Sturm: 
   \newblock {\em Test configurations for K-stability and geodesic rays. }
   \newblock J. Symplectic Geom. \textbf{5} (2007), no. 2, 221--247.

\bibitem[PS10]{PS10}D. H. Phong and J. Sturm:
    \newblock {\em Regularity of geodesic rays and Monge-Amp\`{e}re equations. }
    \newblock Proc. Amer. Math. Soc. \textbf{138} (2010), no. 10, 3637--3650.

\bibitem[RT06]{RT06}J. Ross and R. Thomas: 
   \newblock {\em An obstruction to the existence of constant scalar curvature K\"{a}hler metrics. }
   \newblock  J. Differential Geom. \textbf{72} (2006), no. 3, 429--466.

\bibitem[RWN11]{RWN11}J. Ross and D. Witt Nystr\"{o}m: 
   \newblock {\em Analytic test configurations and geodesic rays.}
   \newblock arXiv:1101.1612. 

\bibitem[Sem92]{Sem92}S. Semmes: 
    \newblock {\em Complex Monge--Amp\`{e}re and symplectic manifolds.}
    \newblock Amer. J. Math. 114 (1992), no. \textbf{3}, 495--550.

%\bibitem[Sz\'{e}10]{Sz\'{e}10}J. Song, S. Zelditch: 
%   \newblock {\em Bergman metrics and geodesics in the space of K\"{a}hler metrics on toric varieties. }
%   \newblock Anal. PDE \textbf{3} (2010), no. 3, 295--358. 

\bibitem[Sto09]{Sto09}J. Stoppa: 
   \newblock {\em K-stability of constant scalar curvature K\"{a}hler manifolds.}
   \newblock Adv. Math. \textbf{221} (2009), no. 4, 1397--1408.

\bibitem[Sz11]{Sz11}G. Sz\'{e}kelyhidi: 
   \newblock  {\em Filtrations and test-configurations.}
   \newblock arXiv:1111.4986. 

\bibitem[Tia97]{Tia97}G. Tian: 
   \newblock {\em K\"{a}hler--Einstein metrics with positive scalar curvature. }
   \newblock   Invent. Math. \textbf{130} (1997), no. 1, 1--37.

\bibitem[WN10]{WN10}D. Witt Nystr\"{o}m: 
   \newblock {\em  Test configurations and Okounkov bodies.}
   \newblock Compositio Math., Available on CJO2012 doi:10.1112/S0010437X12000358. 

\end{thebibliography}
\end{document}